\documentclass[11pt]{article}
\usepackage[english]{babel}
\usepackage{amsmath,amsthm}
\usepackage{amsfonts}
\usepackage[comma,numbers,square,sort&compress]{natbib}
\usepackage{graphicx,subfigure,amsmath,amssymb,mathrsfs,amsfonts,amstext,amsthm}
\usepackage{latexsym, amssymb}
\usepackage{graphicx}
\usepackage{subfigure}
\usepackage{pgf,tikz}
\usepackage{pstricks}
\newtheorem{thm}{Theorem}[section]

\newtheorem{lem}[thm]{Lemma}
\newtheorem{prop}[thm]{Proposition}
\theoremstyle{definition}
\newtheorem{defn}[thm]{Definition}
\theoremstyle{remark}
\newtheorem{rem}[thm]{Remark}
\numberwithin{equation}{section}
\begin{document}

\title{\bfseries\textrm{Finite-Time Elimination of Disagreement of Opinion Dynamics via Covert Noise*} \footnotetext{*This research was supported by the Fundamental Research Funds for the Central Universities (2016JBM070), the National Key Basic Research Program of China (973 program) under grant 2014CB845301/2/3 and the National Natural Science Foundation of China under grants No. 11371049, 61673373 and 91427304.}}
%
\author{Wei Su \thanks{School of Science, Beijing Jiaotong University, Beijing 100044, China, {\tt
suwei@amss.ac.cn}}\and Ge Chen \thanks{National Center for Mathematics and Interdisciplinary Sciences \& Key Laboratory of Systems and
Control, Academy of Mathematics and Systems Science, Chinese Academy of Sciences, Beijing 100190,
China, {\tt
chenge@amss.ac.cn}}\and Yongguang Yu \thanks{School of Science, Beijing Jiaotong University, Beijing 100044, China, {\tt
ygyu@bjtu.edu.cn}}
}

%
\date{}%
\maketitle
\begin{abstract}
Eliminating disagreement in a group is usually beneficial to the social stability. In this paper, using the well-known Hegselmann-Krause (HK) model, we design a quite simple strategy that could  resolve the opinion difference of the system in finite time and induce the opinions synchronized to a targeted value. To be specific, we intentionally introduce weak random noise to only one agent to affect its opinion in the evolution and also a leader agent with fixed opinion in a divisive group, then strictly prove that the disagreement is finally eliminated and the opinions get synchronized to the leader's opinion in finite time. Other than that, we calculate the finite stopping time when the opinions synchronously reach the objective opinion value, which could provide further guide for designing more efficient noise intervention strategies.
\end{abstract}

\textbf{Keywords}: Disagreement elimination, finite stopping time, Hegselmann-Krause model, opinion dynamics, multi-agent systems
\section{Introduction}

In recent years, opinion dynamics has attracted more and more interests of researchers from various areas \cite{Castellano2009,Etesami2015,Jia2015}. To quantitatively investigate the evolution of opinions in a society, several agent-based mathematical models of opinion dynamics have been proposed \cite{DeGroot1974,Friedkin1999,Deffuant2000,Krause2000,Hegselmann2002,Hegselmann2005,Blondel2009,Frasca2013,Jia2015}, whose evolution can be determined by the pre-assumed interaction topology of agents or the confidence bound of the agents. These models effectively characterize the complexity of opinion behaviors, and present various agreement or disagreement phenomena.

One of the central issue in the study of social opinions is the consensus or agreement in the opinion dynamics. However, both the practical observations and the analysis of the mathematical opinion dynamics models show that it is very often the opinions could evolve to be disagreement and the group forms a divisive status.
It is well known that strong disagreement may cause social conflict and result in serious consequences in the society \cite{Nicholson1992,Rahim2001}. Hence, it is usually necessary to enhance the opinion consensus or reduce the opinion difference through some intervention measures. Furthermore, it sometimes needs active opinion intervention to achieve some social goals related to, e.g., opening new market for corporations, commercial negotiations, political elections, or social campaigns, etc. \cite{Bernays1928,Norrander2002}, where opinions are expected to be guided to some desired objectives. Therefore, designing simple and efficient intervention strategies to eliminate the disagreement and induce the opinions to a targeted value is of great importance in opinion dynamics.

Usually it is troublesome and costly to search and design an applicable approach that could eliminate the opinion difference of a group. Fortunately, an interesting phenomenon was found with the confidence-based opinion dynamics models that the random noise in some situations could play a positive role in reducing the disagreement of opinions \cite{Mas2010,Pineda2011,Grauwin2012,Carro2013, Pineda2013}. Here the noise could be the opinion fluctuation caused by rumors, misinformation, news or broadcasts, and even agents' free will. Very recently this discovery was rigorously proved for the well-known Hegselmann-Krause (HK) confidence-based model \cite{Su2016}, where each agent has a bounded confidence, and updates its opinion value by just averaging that of its neighbors (agents whose opinion values
lie within its confidence range). The noise-free HK model will reach static status in finite time and the opinions could be consensus sometimes, but divisive more often. In \cite{Su2016}, it is proved that when all agents of the HK model are affected by random noises during the evolution, the opinions will almost surely reach quasi-consensus (a concept of consensus defined for the noisy case) in finite time and a critical value of noise strength is given. Since the noise is quite easy to generate and manipulate, these findings prompt us to design the disagreement elimination strategy for a divisive group by using random noise. However, the synchronized opinions in \cite{Su2016} are proved to cross the whole opinion space infinite times, while on the other hand we may need the opinions could evolve to some targeted value. In \cite{Han2006,Han2016}, a soft control strategy was proposed by adding a special agent, called ``shill'' which will not destroy the dynamics of the original system, to control the collective behavior intentionally. This inspires us to control the opinions to our objective value.

In this paper, we will design a simple opinion intervention strategy for the HK opinion dynamics by intentionally introducing random noise uniformly distributed on an interval to only one targeted agent in a divisive system to affect its opinion values and also a leader agent who sticks to its opinion value in the system to influence the objective opinions by soft control. Here, for the sake of practical convenience, we only consider to inject weak noise which is covert in manipulation. The noise can be the selective information or standpoints spread intentionally. In this scheme, we will first prove that the separated opinions will almost surely reach $\phi$-consensus (a general concept of consensus) in finite time when the driving noise is adequately weak, then prove that in the presence of leader agent, the synchronized opinions will run to the leader's opinion value in finite time. After that, we will approximately calculate the finite stopping time when the opinions reach $\phi$-consensus and the objective value. It will be shown that the finite stopping time is not integrable when the noise is neutral, and integrable when the noise is oriented (integrable means the expectation of the stopping time is finite). This provides us with further guides for designing more effective noise intervention strategies to eliminate the disagreement of a group.

The organization of this paper is as follows: Section \ref{formulation} formulates the model and the noise intervention scheme, while Sections \ref{quasiconsen}, \ref{Softcontrol} and \ref{calstoptime} give the proof of achieving $\phi$-consensus, controlling opinions to the targeted value and the estimation of the finite stopping time respectively. In Section \ref{simulation}, some numerical simulations of the main results are presented, and the concluding remarks are given in Section \ref{Conclusions}.

\section{Preliminaries and Formulation}\label{formulation}
Suppose there are $n$ agents with $\mathcal{V}=\{1,\ldots,n\}$, whose opinion values at time $t$ are denoted by $x(t)$, which takes values in $(-\infty,\infty)^n$,
then the basic HK model admits the following dynamics:
\begin{equation}\label{basicHKmodel}
  x_i(t+1)=\frac{1}{|\mathcal{N}_i(x(t))|}\sum\limits_{j\in\mathcal{N}_i(x(t))}x_j(t),\,\,\,i\in\mathcal{V},
\end{equation}
where
\begin{equation}\label{neigh}
 \mathcal{N}_i( x(t))=\{j\in\mathcal{V}\; \big|\; |x_j(t)-x_i(t)|\leq \epsilon\}
\end{equation}
is the neighbor set of $i$ and $\epsilon>0$ represents the confidence threshold of the agents. $|\cdot|$ can be the cardinal number of a set or the absolute value of a real number accordingly.

For the noise-free HK model (\ref{basicHKmodel}), it has been proved that the opinion will reach static state in finite time, with agreement or disagreement status:
\begin{prop}\label{HKfrag}\cite{Blondel2009}
For every $1\leq i\leq n$ of (\ref{basicHKmodel}), $x_i(t)$ will converge to some $x_i^*$ in finite time, and either $x_i^*=x_j^*$ or $|x_i^*-x_j^*|> \epsilon$ holds for any $i, j$.
\end{prop}
If $x_i^*=x_j^*$ for all $i,j$, it means the opinions reach consensus and the group realizes agreement. Otherwise, the opinions get separated and the group is divisive. In other words, a divisive opinion system with $N_c$ distinct standpoints can be assumed to satisfy the following conditions:
\begin{equation}\label{divisivesys}
  \left\{
    \begin{array}{ll}
      \mathcal{V}=\bigcup_{g=1}^{N_c}\mathcal{V}_g, & \hbox{~$2\leq N_c\leq n$,} \\
      x_i(0)=x^*_g\in(-\infty,\infty), & \hbox{for~$ i\in\mathcal{V}_g, 1\leq g\leq N_c$,}\\
      x^*_p-x^*_q>\epsilon, & \hbox{for~$1\leq q<p\leq N_c$,}
    \end{array}
  \right.
\end{equation}
where $\{\mathcal{V}_g, 1\le g\le N_c\}$ are the separated subgroups of the divisive system at the initial time and $x_i(0)<x_j(0)$ for any $i\in\mathcal{V}_{g_1}, j\in\mathcal{V}_{g_2}$ with $g_1<g_2$.

In this paper, we focus on a divisive system that satisfies (\ref{divisivesys}) and aim to design a strategy to eliminate the disagreement and synchronize all the opinions.
Usually, it is not easy to implement a simple and practicable method to surely induce the consensus of the divisive opinions. The traditional controller based on the states of the system is hardly feasible for the opinion dynamics since it is almost impossible to acquire the opinion values of all agents and also, the interaction topology of the opinion dynamics is highly dependent on the states. Though it has been proved that random noise could synchronize the opinions in \cite{Su2016}, the noises there are admitted to all agents at the beginning, and this elegant theoretical result still has limitation in applying to design a practicable method to eliminate the opinion difference, since it is hard to guarantee all opinions to be infected and the consensus time is unknown neither. In our intervention scheme of this paper, to eliminate the disagreement of system (\ref{basicHKmodel})-(\ref{divisivesys}), we just consider to inject random noises to only one agent, say agent 1. To be specific, the noise intervention model of (\ref{basicHKmodel})-(\ref{divisivesys}) is as follows:
\begin{equation}\label{noisyHKmodel}
  x_i(t+1)=\frac{1}{|\mathcal{N}_i(x(t))|}\sum\limits_{j\in\mathcal{N}_i(x(t))}x_j(t)+I_{\{i=1\}}\xi(t+1),
\end{equation}
where the noise $\{\xi(t),t\geq 1\}$ are independent and uniformly distributed on
\begin{equation}\label{noisestren}
[-\delta_1,\delta_2],\quad 0\leq\delta_1\leq \delta_2,
\end{equation}
and $I_{\{\cdot\}}$ is the indicator function with $I_{\{i=1\}}=1$ when $i=1$ and $I_{\{i=1\}}=0$ otherwise.
If there is no noise intervention ($\delta_2=0$), the model (\ref{noisyHKmodel}) degenerates to the original model (\ref{basicHKmodel}), and the system stays in divisive status with $N_c$ separated opinions. Here, the random noises can be neutral ($\delta_1=\delta_2$) or oriented ($\delta_1\neq \delta_2$). We take the assumption (\ref{noisestren}) based on the consideration that the noise intervention is added to agent 1 (or any agent in the cluster $\mathcal{V}_1$) whose opinion value is smallest according to (\ref{divisivesys}), then intuitively the effective noise should play a nonnegative role in general to affect its opinion value. If the intervention is added to agent $n$ (or any agent in cluster $\mathcal{V}_{N_c}$), the noise strength can be assumed as $0\leq \delta_2\leq\delta_1$. To study the behavior of model (\ref{neigh})-(\ref{noisestren}), we introduce a definition of consensus in the noisy case:
\begin{defn}\label{robconsen}
Define
\begin{equation*}\label{opindist}
  d_{\mathcal{V}}(t)=\max\limits_{i, j\in \mathcal{V}}|x_i(t)-x_j(t)|~~\mbox{and}~~d_{\mathcal{V}}=\limsup\limits_{t\rightarrow \infty}d_{\mathcal{V}}(t).
\end{equation*}
For $\phi\in[0,\infty)$,\\
(i) If $d_{\mathcal{V}} \leq \phi$, we say the system (\ref{neigh})-(\ref{noisyHKmodel}) will reach $\phi$-consensus.\\
(ii) If $P\{d_{\mathcal{V}} \leq \phi\}=1$, we say almost surely (a.s.) the system (\ref{neigh})-(\ref{noisyHKmodel}) will reach $\phi$-consensus.\\
(iii) Let $T=\inf\{t: d_{\mathcal{V}}(t')\leq \phi \mbox{ for all } t'\geq t\}$.
 If $P\{T<\infty\}=1$, we say a.s. the system (\ref{neigh})-(\ref{noisyHKmodel}) reaches $\phi$-consensus in finite time.
\end{defn}
It is noted that compared with the traditional concept of consensus ($\phi=0$) in the deterministic situation, the $\phi$-consensus here (especially when $\phi$ is small) is more common to describe the agreement of the group, since people rarely hold the completely identical standpoint to the others in a community though they may agree to each other in principle.

\section{$\phi$-consensus of noisy system}\label{quasiconsen}

In this part, we present the main result that the noisy system (\ref{neigh})-(\ref{noisestren}) will reach $\phi$-consensus in finite time:
\begin{thm}\label{quasiconthm}
Consider the noise intervened opinion dynamical system (\ref{neigh})-(\ref{noisestren}), then for
any $\epsilon>0$ and $0<\delta_2\leq \epsilon/2n$, the system will a.s.reach $\delta_2$-consensus in finite time.
\end{thm}
There are several remarks about Theorem \ref{quasiconthm} here compared to the main result Theorem 2 in \cite{Su2016}.
First, in \cite{Su2016}, the initial opinion values $x(0)$ are arbitrarily given, but limited in an bounded interval $[0,1]^n$, while the initial opinion values here are assumed in advance, but extended to an infinite interval. In addition, there is only one agent that is infected by noise here instead of all as in \cite{Su2016}. These variations causes some new mathematical challenges that the analysis skill in \cite{Su2016} fails in proving this theorem, and here we develop a complete new way.
Second, the effective noise strength $0<\delta_2\leq \epsilon/2n$ here is quite conservative compared to that in \cite{Su2016} where $0<\delta_2\leq \epsilon/2$ was obtained. But it is still valuable for practical application, since it is covert and hence more feasible to use tiny noise, such as some minor disturbed information or free information flow, to eliminate the opinion difference. Actually, we conjecture that the effective noise strength here can be extended to $0<\delta_2\leq \epsilon$ which is also critical as in \cite{Su2016}. However, the analysis methods in both \cite{Su2016} and this paper seems to be incapable of completing this.

To prove Theorem \ref{quasiconthm}, the following lemmas are needed:
\begin{lem}\cite{Blondel2009}\label{monosmlem}
Suppose $\{z_i, \, i=1, 2, \ldots\}$ is a nondecreasing (nonincreasing) real sequence.  Then for any integer $s\geq 0$, the sequence
$\{g_s(k)=\frac{1}{k}\sum_{i=s+1}^{s+k}z_i, k\geq 1\}$ is monotonically nondecreasing (nonincreasing) with respect to $k$.
\end{lem}
\begin{lem}\label{robconspeci}
For the system (\ref{neigh})-(\ref{noisestren}), if $0<\delta_2\leq \epsilon$ and there exists a finite time $T\geq 0$ such that $d_\mathcal{V}(T)\leq \epsilon$, then
$d_\mathcal{V}\leq \delta_2$ a.s. on $\{T<\infty\}$.
\end{lem}
\begin{lem}\label{xvalergo}
Still with conditions given in Lemma \ref{robconspeci}, then on $\{T<\infty\}$ it a.s. occurs
\begin{equation*}
  \limsup\limits_{t\rightarrow \infty}x_i(t)= +\infty.
\end{equation*}
\end{lem}
Lemma \ref{robconspeci} states that once all the opinions locate within the region with size of the confidence threshold, the noise with strength no more than $\epsilon$ cannot separate them anymore. Lemma \ref{xvalergo} says that once the system achieves $\delta_2$-consensus, the opinions driven by random noises can assume arbitrarily large region as time tends to infinity no matter how weak the noise is.
The proofs of Lemma \ref{robconspeci} and \ref{xvalergo} are given in Appendix.

\noindent\emph{Proof of Theorem \ref{quasiconthm}.}
At $t=0$, the opinions form $N_c$ separated static clusters as $\mathcal{V}_1,\ldots,\mathcal{V}_{N_c}$, and the noises only affect agent $1\in \mathcal{V}_1$ whose opinion value is at the bottom in the beginning. Let us consider clusters $\mathcal{V}_1$ and $\mathcal{V}_2$. By Lemma \ref{xvalergo}, the opinions of $\mathcal{V}_1$ driven by noises will finally interact with opinions of cluster $\mathcal{V}_2$ at some time, and before that moment, the agents of cluster $\mathcal{V}_1$ possess the same opinion value except agent 1, whose opinion value has an additive $\xi(t)$ according to (\ref{noisyHKmodel}). This means that it is the agent 1 of $\mathcal{V}_1$ who will firstly interact with the agents of $\mathcal{V}_2$. That is, a.s. there is a moment $T_1>0$ such that for all $i\in \mathcal{V}_2$,
$ x_i(T_1)-x_1(T_1)\leq\epsilon.$
Denote
\begin{equation}\label{defstopt}
  \bar{T}=\inf\{t:  x_i(t)-x_1(t)\leq\epsilon,\,\,i\in\mathcal{V}_2\},
\end{equation}
then $P\{\bar{T}<\infty\}=1$ and a.s.
\begin{equation}\label{tfmerge}
\begin{split}
   x_i(\bar{T})-x_1(\bar{T})&\leq\epsilon,\\
   x_i(\bar{T})-x_j(\bar{T})&>\epsilon,\, j\in\mathcal{V}_1-\{1\}
\end{split}
\end{equation}
as well as a.s.
\begin{equation}\label{btfmerge}
d_{\mathcal{V}_1}(\bar{T})\leq \delta_2.
\end{equation}
Then by (\ref{noisyHKmodel})
\begin{equation}\label{nextmerge2}
  \begin{split}
    x_i(\bar{T}+1)= &\frac{1}{1+|\mathcal{V}_2|}\bigg(x_1(\bar{T})+\sum\limits_{k\in \mathcal{V}_2}x_k(\bar{T})\bigg)\\
            = & \frac{1}{1+|\mathcal{V}_2|}x_1(\bar{T})+\frac{|\mathcal{V}_2|}{1+|\mathcal{V}_2|}x_i(\bar{T}),\quad i\in\mathcal{V}_2\\
  \end{split}
\end{equation}
and
\begin{equation}\label{nextmerge1}
  \begin{split}
    x_j(\bar{T}+1) = &\frac{1}{|\mathcal{V}_1|}\bigg(x_1(\bar{T})+\sum\limits_{k\in \mathcal{V}_1-\{1\}}x_k(\bar{T})\bigg)=  \frac{1}{|\mathcal{V}_1|}x_1(\bar{T})+\frac{|\mathcal{V}_1|-1}{|\mathcal{V}_1|}x_j(\bar{T})\\
            =& \frac{1}{|\mathcal{V}_1|}x_1(\bar{T})+\frac{|\mathcal{V}_1|-1}{|\mathcal{V}_1|}(x_1(\bar{T})-\xi(\bar{T}))\\
=&x_1(\bar{T})-\frac{|\mathcal{V}_1|-1}{|\mathcal{V}_1|}\xi(\bar{T}),\quad j\in\mathcal{V}_1-\{1\}
  \end{split}
\end{equation}
and
\begin{equation}\label{nextmerge}
  \begin{split}
  x_1(\bar{T}+1)= &\frac{1}{|\mathcal{V}_1|+|\mathcal{V}_2|}\bigg(\sum\limits_{j\in \mathcal{V}_1}x_j(\bar{T})+\sum\limits_{i\in \mathcal{V}_2}x_2(\bar{T})\bigg)+\xi(\bar{T}+1)\\
            = & \frac{1}{|\mathcal{V}_1|+|\mathcal{V}_2|}(|\mathcal{V}_1|x_1(\bar{T})+|\mathcal{V}_2|x_i(\bar{T}))-\frac{|\mathcal{V}_1|-1}{|\mathcal{V}_1|+|\mathcal{V}_2|}\xi(\bar{T})+\xi(\bar{T}+1)\\
\leq & \frac{1}{1+|\mathcal{V}_2|}(x_1(\bar{T})+|\mathcal{V}_2|x_i(\bar{T}))-\frac{|\mathcal{V}_1|-1}{|\mathcal{V}_1|+|\mathcal{V}_2|}\xi(\bar{T})+\xi(\bar{T}+1)\\
\leq & x_1(\bar{T})+\frac{|\mathcal{V}_2|}{1+|\mathcal{V}_2|}\epsilon-\frac{|\mathcal{V}_1|-1}{|\mathcal{V}_1|+|\mathcal{V}_2|}\xi(\bar{T})+\xi(\bar{T}+1),\quad i\in\mathcal{V}_2,\\
  \end{split}
\end{equation}
where the last two inequalities follow from $|\mathcal{V}_1|\geq 1$, (\ref{tfmerge}) and Lemma \ref{monosmlem}.
By (\ref{nextmerge2}) and (\ref{nextmerge1}), we have a.s.
\begin{equation}\label{dismerge1}
\begin{split}
  x_i(\bar{T}+1)-x_j(\bar{T}+1)  =& \frac{|\mathcal{V}_2|}{1+|\mathcal{V}_2|}(x_i(\bar{T})-x_1(\bar{T}))+\frac{|\mathcal{V}_1|-1}{|\mathcal{V}_1|}\xi(\bar{T})\\
    \leq & \frac{|\mathcal{V}_2|}{1+|\mathcal{V}_2|}\epsilon+\frac{|\mathcal{V}_1|-1}{|\mathcal{V}_1|}\delta_2,\,i\in\mathcal{V}_2,\,j\in\mathcal{V}_1-\{1\}\\
\end{split}
\end{equation}
and by (\ref{nextmerge1}), (\ref{nextmerge}), we have a.s.
\begin{equation}\label{dismerge2}
\begin{split}
 x_1(\bar{T}+1)-x_j(\bar{T}+1)\leq &\frac{|\mathcal{V}_1|-1}{|\mathcal{V}_1|}\xi(\bar{T})+ \frac{|\mathcal{V}_2|}{1+|\mathcal{V}_2|}\epsilon-\frac{|\mathcal{V}_1|-1}{|\mathcal{V}_1|+|\mathcal{V}_2|}\xi(\bar{T})\\
  &+\xi(\bar{T}+1)\\
  \leq &\frac{n-1}{n}\epsilon+(1-\frac{1}{n})\delta_2+\delta_2\\
= &\frac{n-1}{n}\epsilon+(2-\frac{1}{n})\delta_2,
\end{split}
\end{equation}
since $\frac{1}{|\mathcal{V}_1|}+\frac{|\mathcal{V}_1|-1}{|\mathcal{V}_1|+|\mathcal{V}_2|}\geq \frac{1}{n}$ holds for $1\leq|\mathcal{V}_1|,|\mathcal{V}_2|\leq n, |\mathcal{V}_1|+|\mathcal{V}_2|\leq n$.
If $\delta_2\leq \epsilon/2n$, by (\ref{dismerge1}) and (\ref{dismerge2}), it has a.s.
\begin{equation}\label{merge}
  x_i(\bar{T}+1)-x_j(\bar{T}+1)\leq \epsilon, \quad x_1(\bar{T}+1)-x_j(\bar{T}+1)\leq \epsilon.
\end{equation}
(\ref{merge}) implies that at $\bar{T}+1$, all agents of $\mathcal{V}_1$ and $\mathcal{V}_2$ are neighbors to each other, then by Lemma \ref{robconspeci}, the opinions of $\mathcal{V}_1$ and $\mathcal{V}_2$ reach $\delta_2$-consensus in finite time. Repeating the above procedure with respect to clusters of $\mathcal{V}_3$ to $\mathcal{V}_{N_c}$, we can get that all opinions a.s. reach $\delta_2$-consensus in finite time. This completes the proof. \hfill $\Box$

\section{Soft control strategy}\label{Softcontrol}
In the last section, we show that weak noise can eliminate the disagreement of opinions and induce the consensus of the system. However, Lemma \ref{xvalergo} shows that the synchronized opinions driven by random noise could go to infinity, which contradicts with our practical wish to control opinions to some objective value. In this part, we will design a simple soft control strategy to guide the opinions to the targeted opinion value. In this soft control strategy, we introduce in the system a ``leader'' agent who is closed-minded and sticks to the fixed opinion value, but other agents will take into account the leader's opinion once the leader locates in their neighborhood region. Then we prove that the opinions will almost surely gather to the leader's opinion in finite time. To demonstrate the sketch of the above soft control scheme, denote the leader as agent $\textit{l}$, and we simply set $\textit{l}$'s opinion value to be
\begin{equation}\label{leadervalue}
 x_\textit{l}(t)=A>x_{N_c}^*+\epsilon,\quad t\geq 0
\end{equation}
in the divisive system (\ref{neigh})-(\ref{noisestren}), then we prove
\begin{thm}\label{softcontthm}
For the divisive system (\ref{neigh})-(\ref{noisestren}) with an additional leader agent $\textit{l}$ whose opinion value satisfies (\ref{leadervalue}), still with the conditions given in Theorem \ref{quasiconthm}, and denote $d_\mathcal{V}^A(t)=\max\limits_{i\in\mathcal{V}}|x_i(t)-A|$, then a.s. $\limsup\limits_{t\rightarrow\infty}d_\mathcal{V}^A(t)\leq 2\delta_2$.
\end{thm}
\begin{rem}
Here we only demonstrate a simple scheme by setting the leader above all the agents for convenience. Actually, a noise-free HK opinion system with closed-minded agents has been shown to converge (maybe not in finite time) \cite{Chazelle2015}. If we consider the case when the leader agent can be placed anywhere, such as in the intermediate region of the divisive opinions, the disagreement elimination strategy may not work by introducing noise to only one agent, when more flexible strategies can be designed based on the spirit of noise intervention here, such as more agents can be considered to receive noise.
\end{rem}
\noindent\emph{Proof of Theorem \ref{softcontthm}.}
Theorem \ref{quasiconthm} tells that the opinions of all agents in $\mathcal{V}$ will reach $\delta_2$-consensus in finite time $T$, and suppose $T=0$ a.s. without loss of generality.
If $x_i(t)<A-\epsilon, i\in\mathcal{V}$, it is easy to get that a.s.
\begin{equation}\label{smallthanA}
  |x_1(t+1)-x_j(t)|\leq \delta_2\leq \epsilon,\, j\in\mathcal{V}-\{1\}.
\end{equation}
If $\max\{d_\mathcal{V}(t),d_\mathcal{V}^A(t) \}\leq\epsilon$ a.s., we have for $i\in\mathcal{V}$,
\begin{equation}\label{neighinA}
 \begin{split}
  x_i(t+1)=&\frac{1}{n+1}\bigg(A+\sum\limits_{k\in\mathcal{V}}x_k(t)\bigg)+I_{\{i=1\}}\xi(t+1)\\
  =&\frac{1}{n+1}A+\frac{n}{n+1}x_j(t)+\frac{1}{n+1}\xi(t)+I_{\{i=1\}}\xi(t+1),\,\,\,j\in\mathcal{V}-\{1\}
  \end{split}
\end{equation}
then a.s.
\begin{equation}\label{distancetoA}
\begin{split}
d_\mathcal{V}(t+1)\leq& |\xi(t+1)|\leq \delta_2\leq \epsilon,\\
  |A-x_i(t+1)|\leq&|\frac{n}{n+1}(A-x_i(t))|+\frac{|\xi(t)|}{n+1}+|\xi(t+1)|\\
  \leq & \frac{n}{n+1}\epsilon+\frac{n+2}{n+1}\delta_2\leq \epsilon,\,\,\,i\in\mathcal{V}
\end{split}
\end{equation}
since $\delta_2\leq \frac{\epsilon}{2n}$. Repeating the above procedure, it obtains that once (\ref{distancetoA}) holds at some moment, it holds forever. Next we will show that $d_\mathcal{V}(t)\leq \epsilon, t\geq 0$ a.s.. If not, by (\ref{smallthanA}) and (\ref{distancetoA}), it is easy to know that it could only happen when $x_1(t)\in[A-\epsilon, A)$ and  $x_j(t)\in[A-\epsilon-\delta_2,A-\epsilon]$. In this case
\begin{equation*}
\begin{split}
  x_1(t+1)=&\frac{1}{n+1}\bigg(A+x_1(t)+\sum\limits_{j\in\mathcal{V}-\{1\}}x_j(t)\bigg)+\xi(t+1)\\
  =&\frac{A+x_1(t)}{n+1}+\frac{n-1}{n+1}x_j(t)+\xi(t+1)\\
  x_j(t+1)=&\frac{1}{n}\bigg(x_1(t)+\sum\limits_{k\in\mathcal{V}-\{1\}}x_k(t)\bigg)\\
  =&\frac{n-1}{n}x_j(t)+\frac{1}{n}x_1(t),\quad j\in\mathcal{V}-\{1\},
\end{split}
\end{equation*}
yielding
\begin{equation}\label{distan1j}
\begin{split}
  x_1(t+1)-x_j(t+1)=&\frac{A-x_j(t)}{n+1}-\frac{1}{n(n+1)}(x_1(t)-x_j(t))+\xi(t+1)\\
  \leq& \frac{\epsilon+\delta_2}{n+1}+\delta_2\leq \epsilon,\,a.s..
\end{split}
\end{equation}
Hence
\begin{equation}\label{distan1jever}
 d_\mathcal{V}(t)\leq\epsilon, \,t\geq 0\quad a.s..
\end{equation}
Denote the original opinion values that satisfy (\ref{neigh})-(\ref{noisestren}) without introducing the leader agent as $y_i(t), i\in\mathcal{V}, t\geq 1$, and
\begin{equation}\label{stopreachleader}
  T_l=\inf\{t:d_\mathcal{V}^A(t)\leq \epsilon\},
\end{equation}
then by Lemma \ref{monosmlem} and (\ref{distan1jever}), we have
\begin{equation}\label{xyvalue}
  x_i(t)\geq y_i(t),\,\,a.s.,\, t\leq T_l.
\end{equation}
By Lemma \ref{xvalergo}, we have
\begin{equation}\label{Tlprob}
  P\{1\leq T_l<\infty\}=1.
\end{equation}
By (\ref{distancetoA}), (\ref{distan1jever}) and (\ref{Tlprob}),
\begin{equation}\label{distanVA}
  \max\{d_\mathcal{V}(t),d_\mathcal{V}^A(t)\}\leq \epsilon,\,\,a.s.\,\,t\geq T_l.
\end{equation}
Then on $\{1\leq T_l<\infty\}$, for $j\in\mathcal{V}-\{1\}, t\geq T_l,$ by (\ref{neighinA})
\begin{equation}\label{distancetoA2}
\begin{split}
  |A-x_j(t+1)|\leq&\frac{n}{n+1}|A-x_j(t)|+\frac{1}{n+1}|\xi(t)|\\
  \leq& \bigg(\frac{n}{n+1}\bigg)^2|A-x_j(t-1)|+\frac{n}{n+1}\frac{|\xi(t-1)|}{n+1}+\frac{|\xi(t)|}{n+1}\\
  &\cdots\\
  \leq & \bigg(\frac{n}{n+1}\bigg)^{t-T_l+1}|A-x_j(T_l)|+\sum\limits_{k=T_l}^t\bigg(\frac{n}{n+1}\bigg)^{t-k}\frac{|\xi(k)|}{n+1}\\
  \leq& \bigg(\frac{n}{n+1}\bigg)^{t-T_l+1}|A-x_j(T_l)|\\
  &+ (n+1)\bigg(1-\bigg(\frac{n}{n+1}\bigg)^{t-T_l+1}\bigg)\frac{\delta_2}{n+1},\,a.s.,
\end{split}
\end{equation}
whence
\begin{equation*}
  \limsup\limits_{t\rightarrow \infty}|A-x_j(t+1)|\leq \delta_2,\,\,a.s.,\,j\in\mathcal{V}-\{1\}
\end{equation*}
and a.s.
\begin{equation*}
\begin{split}
   \limsup\limits_{t\rightarrow \infty}|A-x_1(t+1)|\leq& \limsup\limits_{t\rightarrow \infty}(|A-x_j(t+1)|+|x_1(t+1)-x_j(t+1)|)\\
   \leq& 2\delta_2,
   \end{split}
\end{equation*}
yielding the conclusion. \hfill $\Box$

\section{Stopping time of reaching $\phi$-consensus}\label{calstoptime}
Theorem \ref{quasiconthm} tells that the divisive system (\ref{neigh})-(\ref{noisestren}) can a.s. achieve $\delta_2$-consensus in finite time under the drive of the weak random noise, and Theorem \ref{softcontthm} shows that the opinions will run to the objective value in finite time with the soft control of a leader agent. Then the next interest is what the finite time is. Actually, by the proof of Theorem \ref{quasiconthm}, the time $T$ when the noisy system (\ref{neigh})-(\ref{noisestren}) reach $\delta_2$-consensus is a stopping time which is finite with $P\{T=\infty\}=0$. For a finite stopping time, its expectation could be finite or infinite, then with the conditions given in Theorem \ref{quasiconthm}, we have
\begin{thm}\label{stoptimethm}
Let $T=\inf\{t: d_{\mathcal{V}}(t)\leq \epsilon\}$, then
\begin{enumerate}
  \item $\textbf{E}\,T=\infty$, if $\delta_1=\delta_2$;
  \item $\textbf{E}\,T\leq  \frac{2n}{\delta_2-\delta_1}(x_{N_c}^*-x_1^*+(N_c-1)\delta_2)$, if $\delta_1<\delta_2$.
\end{enumerate}
\end{thm}
\begin{rem}
Usually, we say a finite stopping time $T$ is integrable if $\textbf{E}\,T<\infty$, then Theorem \ref{stoptimethm} implies that the finite stopping time when the opinion differences are eliminated is integrable only when the noise is oriented. Also, it can be seen from the coefficient $\frac{2n}{\delta_2-\delta_1}$ that given $0<\delta_2\leq \epsilon/2n$, a tiny increase of $\delta_1$ could dramatically decrease the mean of the stopping time, which will be illustrated by the simulations in the next section.
\end{rem}
To prove Theorem \ref{stoptimethm}, some preliminary lemmas are needed.
\begin{lem}\label{stoptimejud}
 Suppose $\{X_t,t\geq 1\}$ are i.i.d. random variables with $\textbf{E}\,X_1=\mu\geq 0, \textbf{E}\,|X_1|>0$, and let $S_t=\sum_1^tX_i, T_0=\inf\{t\geq 1: S_t\geq 0\}, T_c=\inf\{t\geq 1: S_t>c>0\}$, then $\textbf{E}\,T_c\geq\textbf{E}\,T_0=\infty$ for $\mu=0$ and $\textbf{E}\,T_c<\infty$ for $\mu>0$.
\end{lem}
\begin{proof}
Consider Exercise 5.4.11 of \cite{Chow1997} (Ex.11), then when $\mu=0$, we have $\textbf{E}\,T_c=\infty$, whence $\textbf{E}\,T_0=\infty$ by Theorem 5.4.1 of \cite{Chow1997}. Apropos of the case of $\mu>0$, denote $T_0^+=\inf\{t\geq 1: S_t> 0\}$, then $\textbf{E}\,T_0^+<\infty$ by Ex.11, and hence $\textbf{E}\,T_0\leq \textbf{E}\,T_0^+<\infty$, which implies $\textbf{E}\,T_c<\infty$ by Theorem 5.4.1 of \cite{Chow1997}.
\end{proof}
\begin{lem}\label{stoptimejudalp}
Suppose $\{X_t,t\geq 1\}$ are i.i.d. random variables uniformly distributed on $[-\delta,\delta]$ with $\delta>0$, and denote $Q_1=X_1, Q_t=\sum_1^{t-1}X_i+\alpha X_t, t\geq 2$ with $\alpha> 1$. Let $ T_0^{'}=\inf\{t\geq 1: Q_t\geq 0\}$, then $\textbf{E}\,T_0^{'}=\infty$.
\end{lem}
It is easy to check that $T_0^{'}$ is shorter than $T_0$ defined in Lemma \ref{stoptimejud}, however, Lemma \ref{stoptimejudalp} shows that $T_0^{'}$ is still not integrable. The proof of Lemma \ref{stoptimejudalp} is given in Appendix.
\begin{lem}\label{waldequ}\cite{Chow1997}
(Wald's Equation). Let $\{X_t,t\geq 1\}$ are i.i.d. random variables, and $S_t=\sum_1^tX_i, t\geq 1$. If $\textbf{E}\,X_1$ exists and $T$ is an $\{X_n\}-$stopping time with $\textbf{E}\,T<\infty$, then
\begin{equation*}
  \textbf{E}\,S_T=\textbf{E}\,X_1\cdot\textbf{E}\,T.
\end{equation*}
\end{lem}

\noindent\emph{Proof of Theorem \ref{stoptimethm}.}
Considering the stopping time $\bar{T}$ defined in (\ref{defstopt}), it is known from the proof of Theorem \ref{robconsen} that at time $\bar{T}+1$, the clusters $\mathcal{V}_1$ and $\mathcal{V}_2$ achieve $\delta_2$-consensus. Hence, first we need to calculate the stopping time $\bar{T}$. By (\ref{defstopt}) and (\ref{tfmerge}), we can know that at time $\bar{T}$, only agent 1 of group $\mathcal{V}_1$ becomes the neighbor of the agents in group $\mathcal{V}_2$. Then for $0< t\leq\bar{T}$,
\begin{equation}\label{ytevolu}
\begin{split}
  x_1(t) =&\frac{1}{|\mathcal{V}_1|}\sum\limits_{j\in \mathcal{V}_1}x_j(t-1)+\xi(t)\\
  = &\frac{1}{|\mathcal{V}_1|}\bigg(\sum\limits_{j\in \mathcal{V}_1}\frac{1}{|\mathcal{V}_1|}\sum\limits_{k\in \mathcal{V}_1}x_k(t-2)+\xi(t-1)\bigg)+\xi(t)\\
  =& \frac{1}{|\mathcal{V}_1|}\sum\limits_{j\in \mathcal{V}_1}x_j(t-2)+\frac{1}{|\mathcal{V}_1|}\xi(t-1)+\xi(t) \\
    =&\frac{1}{|\mathcal{V}_1|}\sum\limits_{j\in \mathcal{V}_1}x_j(0)+\sum\limits_{k=1}^{t-1}\frac{1}{|\mathcal{V}_1|}\xi(k)+\xi(t)\\
=&x_1^*+\frac{1}{|\mathcal{V}_1|}\sum\limits_{k=1}^{t-1}\xi(k)+\xi(t).
\end{split}
\end{equation}
Let $R_1=\xi(1),R_t=\sum\limits_{k=1}^{t-1}\xi(k)+|\mathcal{V}_1|\xi(t),t\geq 2$, then by (\ref{defstopt}), (\ref{tfmerge}) and (\ref{ytevolu}),
\begin{equation}\label{barTredef}
  \bar{T}=\inf\{t>0:R_t\geq |\mathcal{V}_1|(x_2^*-x_1^*-\epsilon)\}.
\end{equation}

(i) $\delta_1=\delta_2$: If $|\mathcal{V}_1|=1$, by (\ref{barTredef}) and Lemma \ref{stoptimejud}, it can be obtained that $\textbf{E}\,\bar{T}=\infty$. Otherwise $|\mathcal{V}_1|>1$, then $\textbf{E}\,\bar{T}=\infty$ can be obtained by (\ref{barTredef}) and Lemma \ref{stoptimejudalp}. Hence
\begin{equation}\label{stoptimesymm}
  \textbf{E}\,T\geq \textbf{E}\,\bar{T}=\infty.
\end{equation}

(ii) $\delta_1<\delta_2$: Denote $T_1=\inf\{t>0:U_t\geq M+(|\mathcal{V}|-1)\delta_2\}$.
Since $\{R_t<M\}\subset\{U_t<M+(|\mathcal{V}|-1)\delta_2\}$, we have
\begin{equation}\label{probleq}
\begin{split}
  P\{\bar{T}\geq t+1\}=& P\{\max\limits_{j\leq t}R_j<M\}\leq P\{\max\limits_{j\leq t}U_j<M+(|\mathcal{V}|-1)\delta_2\}\\
 =& P\{T_1\geq t+1\},\quad t\geq 1.
  \end{split}
\end{equation}
Note that $\textbf{E}\,\xi(1)=\frac{\delta_2-\delta_1}{2}>0$, then by Lemma \ref{stoptimejud}, $\textbf{E}\,T_1<\infty$. Thus by (\ref{probleq})
\begin{equation}\label{expectleq}
\begin{split}
  \textbf{E}\,\bar{T}=&\sum\limits_{t=1}^{\infty}P\{\bar{T}\geq t\}\leq \sum\limits_{t=1}^{\infty}P\{T_1\geq t\}= \textbf{E}\,T_1<\infty.
  \end{split}
\end{equation}
Since $U_{T_1-1}<M+(|\mathcal{V}|-1)\delta_2, \xi(1)\leq \delta_2$, it yields
\begin{equation}\label{UTvalue}
  U_{T_1}\leq M+|\mathcal{V}|\delta_2,\,\,a.s..
\end{equation}
By Lemma \ref{waldequ} and (\ref{UTvalue})
\begin{equation}\label{expectT1}
  \textbf{E}\,T_1=\frac{\textbf{E}\,U_{T_1}}{\textbf{E}\,\xi(1)}\leq\frac{2|\mathcal{V}|(x_2^*-x_1^*-\epsilon+\delta_2)}{\delta_2-\delta_1}.
\end{equation}
Hence (\ref{expectleq}) yields
\begin{equation}\label{expetbarT}
   \textbf{E}\,\bar{T}\leq\frac{2|\mathcal{V}|(x_2^*-x_1^*-\epsilon+\delta_2)}{\delta_2-\delta_1}.
\end{equation}
Define $T^{(1)}=T_1, T^{(j+1)}=\inf\{t\geq1: U_{T_j+t}-U_{T_j}\geq M+(|\mathcal{V}|-1)\delta_2\},\,j\geq 1$ where $T_j=\sum_1^jT^{(i)}$, then $\{T^{(j)},j\geq  1\}$ are the copies of $T_1$.
Since $\bar{T}$ is the time when the clusters $\mathcal{V}_1$ and $\mathcal{V}_2$ get merged, following the above procedure, it obtains the time $T$ when all the clusters achieve $\delta_2$-consensus
\begin{equation}\label{expectT}
\begin{split}
  \textbf{E}\,T\leq& \frac{2|\mathcal{V}|}{\delta_2-\delta_1}\sum\limits_{j=1}^{N_c-1}(x_{j+1}^*-x_j^*-\epsilon+\epsilon+\delta_2)\\
  =&\frac{2n}{\delta_2-\delta_1}(x_{N_c}^*-x_1^*+(N_c-1)\delta_2).
  \end{split}
\end{equation}
This complete the proof. \hfill$\Box$

For the finite stopping time when the opinions are synchronized to the leader's opinion value, it can be defined as $T_l$ by considering (\ref{stopreachleader}) and (\ref{distanVA}). Compared to the stopping time $T$ when all the opinions reach $\delta_2$-consensus, $T_l$ increases by the time amount of reaching to $A$ for the synchronized opinions. Then by (\ref{stoptimesymm}) and (\ref{expectT}), we have
\begin{thm}\label{stoptimetoleaderthm}
Let $T_l=\inf\{t:d_\mathcal{V}^A(t)\leq \epsilon\}$ with $d_\mathcal{V}^A(t)=\max\limits_{i\in\mathcal{V}}|x_i(t)-A|$, then
\begin{enumerate}
  \item $\textbf{E}\,T_l=\infty$, if $\delta_1=\delta_2$;
  \item $\textbf{E}\,T_l\leq \frac{2n}{\delta_2-\delta_1}(A-x_1^*+\epsilon+N_c\delta_2)$, if $\delta_1<\delta_2$.
\end{enumerate}
\end{thm}

\section{Simulations}\label{simulation}
In this part, we will present some simulation results to illustrate the main theorems. Here, we suppose the agent number is 10, then by Theorem \ref{quasiconthm}, we can get that the system (\ref{neigh})-(\ref{noisestren}) will reach $\delta_2$-consensus in finite time when $0<\delta_2\leq 0.05\epsilon$. And by Theorem \ref{stoptimejud}, we know that the stopping time of reaching $\delta_2$-consensus is not integrable when $0<\delta_1=\delta_2$, and integrable when $\delta_1<\delta_2$. Figure \ref{symmnoisefig} and \ref{biasnoisefig} illustrate that the system (\ref{neigh})-(\ref{noisestren}) achieve $\delta_2$-consensus in finite time when the noises are either neutral or oriented, but the stopping time of reaching $\delta_2$-consensus in the neutral case (Figure \ref{symmnoisefig}) is much longer than that in the oriented case (Figure \ref{biasnoisefig}). Other than that, a large number of simulations show that only a tiny increase of $\delta_1$ will sharply decrease the time quantity when the system reach $\delta_2$-consensus, and this can also be predicted from Theorem \ref{stoptimejud}. Figure \ref{symmleaderfig} and \ref{biasleaderfig} illustrate the results when a leader agent is added to the system, and it can be seen that the opinions are synchronized to the leader's opinion in finite time. Also, the time of gathering to the leader's opinion when the noise is oriented is shown to be much shorter than the neutral noise case even though $\delta_1$ changes very little.
\begin{figure}[htp]
  \centering
  \includegraphics[width=4in]{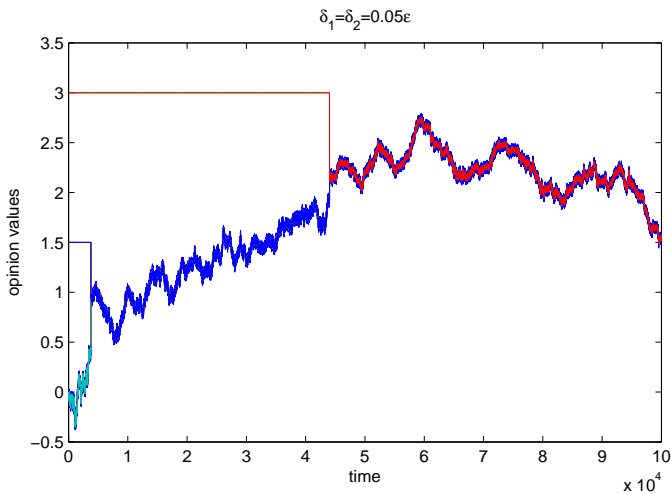}\\
  \caption{Opinion evolution of system (\ref{neigh})-(\ref{noisestren}) of 10 agents with neutral noise uniformly distributed on $[-0.05,0.05]\epsilon$. The initial opinion value $x(0)=[0 ~0 ~0 ~0 ~1.5 ~1.5~ 1.5~ 1.5 ~3 ~3]$, confidence threshold $\epsilon=1$, noise strength $\delta_1=\delta_2=0.05$.  }\label{symmnoisefig}
\end{figure}
\begin{figure}[htp]
  \centering
  \includegraphics[width=4in]{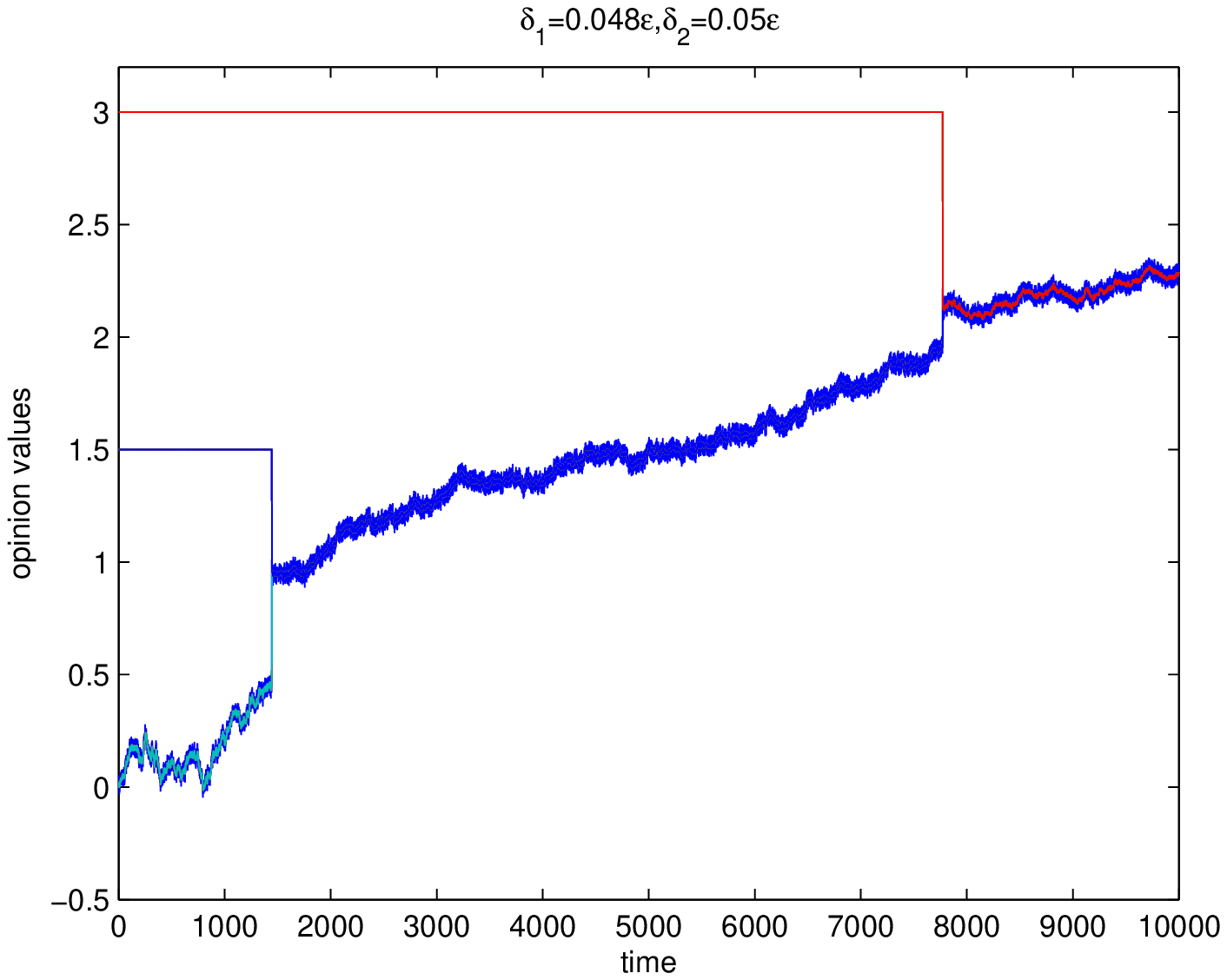}\\
  \caption{Opinion evolution of system (\ref{neigh})-(\ref{noisestren}) of 10 agents with oriented noise uniformly distributed on $[-0.048,0.05]\epsilon$. The initial opinion value $x(0)=[0 ~0 ~0 ~0 ~1.5 ~1.5~ 1.5~ 1.5 ~3 ~3]$, confidence threshold $\epsilon=1$, noise strength $\delta_1=0.048,\delta_2=0.05$.  }\label{biasnoisefig}
\end{figure}
\begin{figure}[htp]
  \centering
  \includegraphics[width=4in]{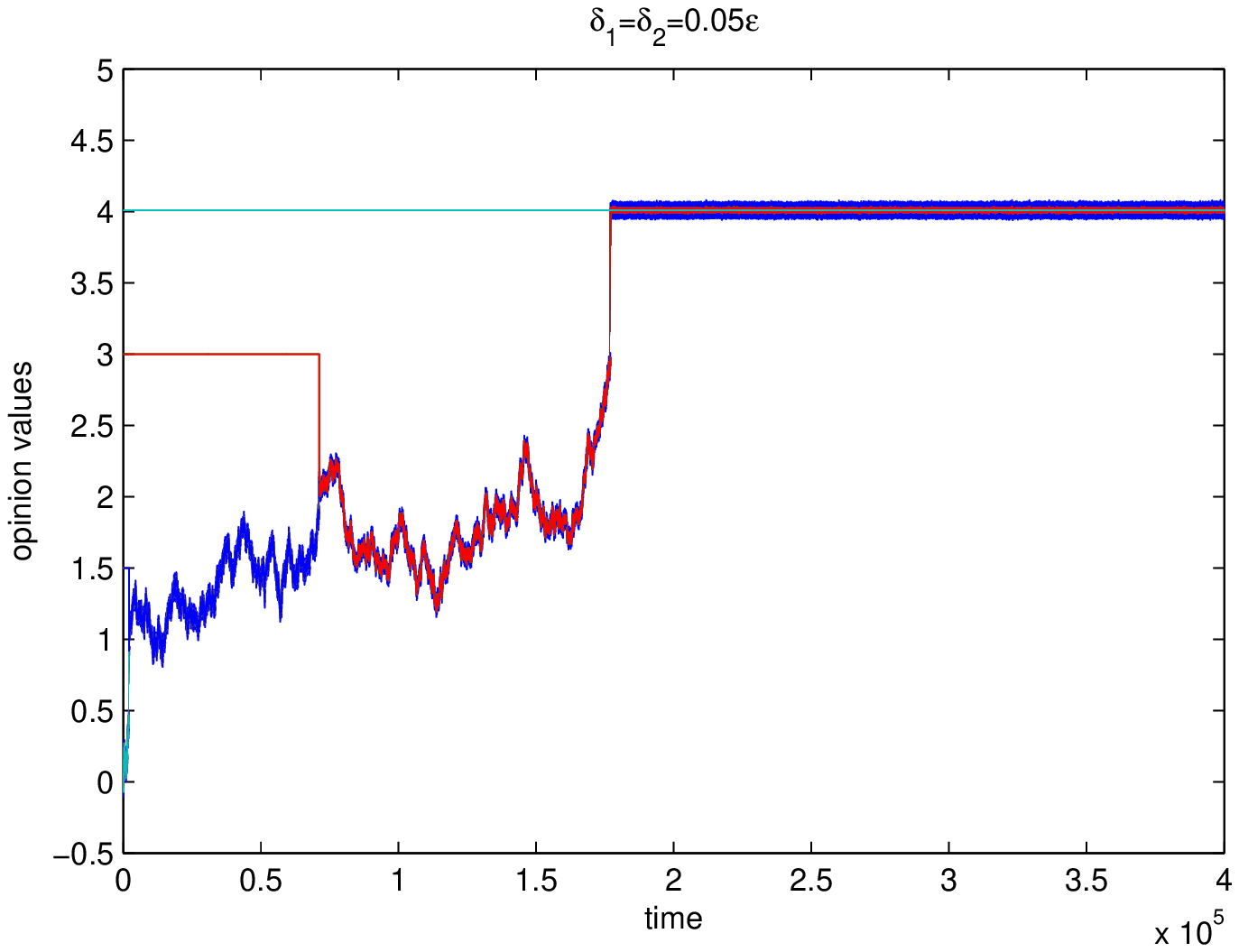}\\
  \caption{Opinion evolution of system (\ref{neigh})-(\ref{noisestren}) of 11 agents with neutral noise uniformly distributed on $[-0.05,0.05]\epsilon$. The initial opinion value $x(0)=[0 ~0 ~0 ~0 ~1.5 ~1.5~ 1.5~ 1.5 ~3 ~3 ~4.01]$, confidence threshold $\epsilon=1$, noise strength $\delta_1=\delta_2=0.05$.  }\label{symmleaderfig}
\end{figure}
\begin{figure}[htp]
  \centering
  \includegraphics[width=4in]{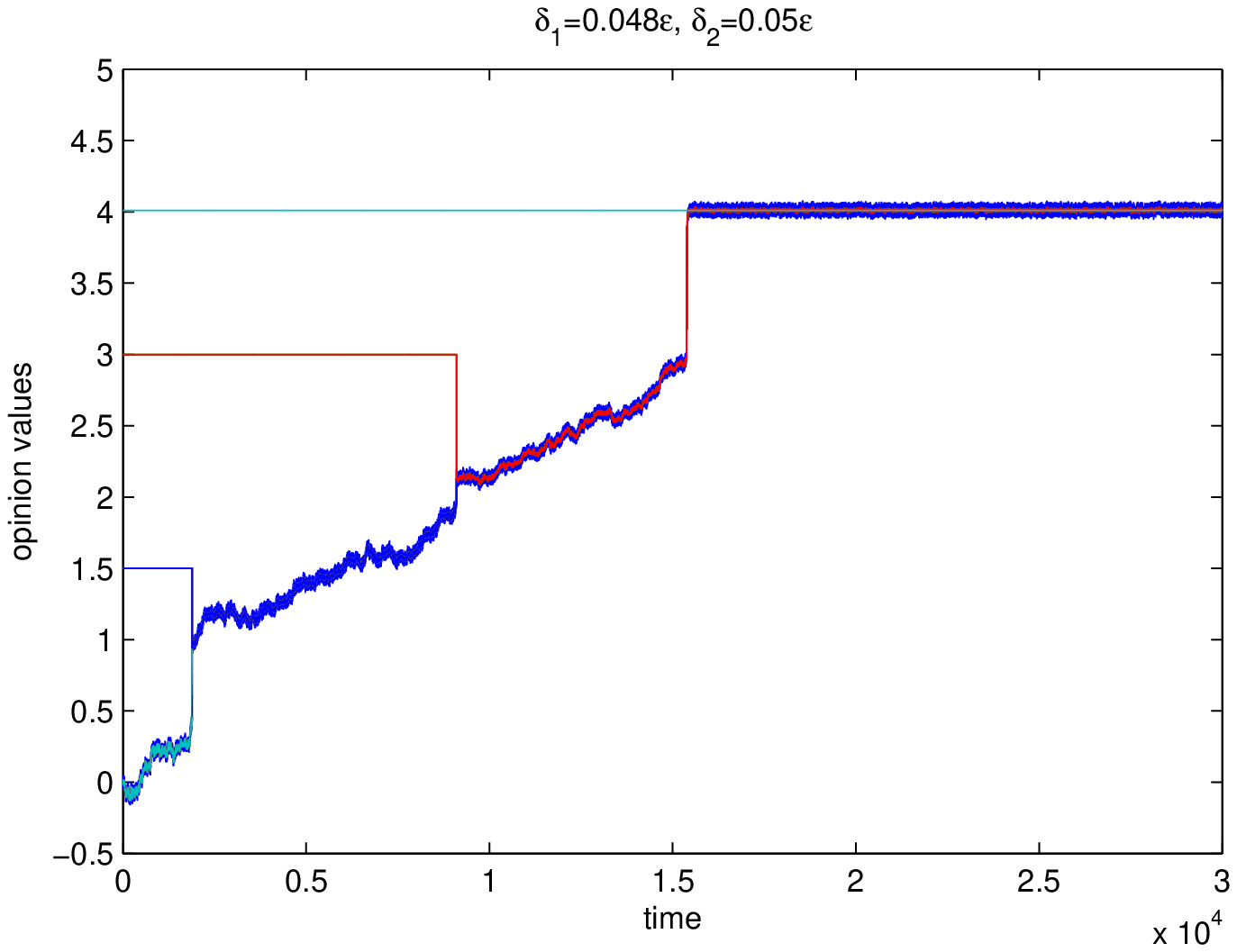}\\
  \caption{Opinion evolution of system (\ref{neigh})-(\ref{noisestren}) of 11 agents with oriented noise uniformly distributed on $[-0.048,0.05]\epsilon$. The initial opinion value $x(0)=[0 ~0 ~0 ~0 ~1.5 ~1.5~ 1.5~ 1.5 ~3 ~3 ~4.01]$, confidence threshold $\epsilon=1$, noise strength $\delta_1=0.048,\delta_2=0.05$.  }\label{biasleaderfig}
\end{figure}
\section{Conclusions}\label{Conclusions}
In this paper, we have proposed a simple intervention strategy to eliminate the disagreement of a divisive system and further induce the opinions synchronized to an objective value, by introducing covert weak random noise to one targeted individual and a leader agent to the system. First we proved that the proposed measure can almost surely synchronize the divisive opinions to the leader's opinion in finite time. Then we calculated the finite stopping time when the noisy opinions achieve $\phi$-consensus. It is interesting to find that the stopping time is not integrable when noise is neutral, and integrable when the  noise is oriented. This fact will provide us with further insights for designing more efficient strategies to eliminate the social disagreement in practice.

\appendix
\section{Proof of Lemma \ref{robconspeci}}
\begin{proof}
 At time $T$, all agents are neighbors to each other, then by (\ref{noisyHKmodel}),
\begin{equation*}
  x_i(T+1)=\frac{1}{n}\sum\limits_{j=1}^{n}x_j(T)+I_{\{i=1\}}\xi(T+1),
\end{equation*}
implying $d_{V}(T+1)=|\xi(T+1)|\leq \delta_2, a.s.$. Repeating the procedure yields the conclusion.
\end{proof}
\section{Proof of Lemma \ref{xvalergo}}
\begin{proof}
  Suppose $T=0$ a.s. without loss of generality, then
 \begin{equation}\label{xtevolu}
\begin{split}
 x_i(t+1)=&\frac{1}{n}\sum\limits_{j=1}^{n}x_j(t)+I_{\{i=1\}}\xi(t+1)\\
=& \frac{1}{n}\sum\limits_{j=1}^{n}\bigg(\frac{1}{n}\sum\limits_{k=1}^{n}x_k(t-1)+\xi(t)\bigg)+I_{\{i=1\}}\xi(t+1)\\
  =& \frac{1}{n}\sum\limits_{j=1}^{n}x_j(t-1)+\frac{1}{n}\xi(t)+I_{\{i=1\}}\xi(t+1) \\
 =&\frac{1}{n}\sum\limits_{j=1}^{n}x_j(0)+\sum\limits_{k=1}^{t}\frac{1}{n}\xi(k)+I_{\{i=1\}}\xi(t+1).
\end{split}
\end{equation}
Let $S_t=\frac{1}{n}\sum\limits_{k=1}^{t-1}\xi(k),t\geq 2$, then by (\ref{xtevolu}),
\begin{equation}\label{rewritext}
  x_i(t+1)=\frac{1}{n}\sum\limits_{j=1}^{n}x_j(0)+S_t+I_{\{i=1\}}\xi_i(t+1).
\end{equation}
Since $\{\xi(t),t\geq 1\}$ are i.i.d. random variables with $\textbf{E}\,\xi(1)=\frac{\delta_2-\delta_1}{2}$, by Law of Large Number, $S_t/t\rightarrow \frac{\delta_2-\delta_1}{2n}$ a.s.. If $\delta_1<\delta_2$, we have $\limsup\limits_{t\rightarrow \infty}S_t\rightarrow \infty$ a.s.. Otherwise if $\delta_1=\delta_2$, it is from Theorem 5.4.3 of \cite{Chow1997} that $\limsup\limits_{t\rightarrow \infty}S_t\rightarrow \infty$ a.s..
Then by (\ref{rewritext}), $\limsup\limits_{t\rightarrow \infty}x_i(t)\rightarrow \infty$ a.s..
\end{proof}
\section{Proof of Lemma \ref{stoptimejudalp}}
The following result is presented as Exercise 5.4.8 in \cite{Chow1997}:
\begin{lem}\label{ex548}
Suppose $S_t=\sum_1^tX_j$, where $\{X_t, t\geq 1\}$ are nondegenerate i.i.d. random variables, and let $T=\inf\{j\geq 1: S_j<-a<0\,\,\text{or}\,\,S_j>b>0\}$, then $\textbf{E}\,T<\infty.$
\end{lem}
Now we give the proof of Lemma \ref{stoptimejudalp}:
\begin{proof}
Let $S_0=0, S_t=\sum_i^tX_j, t\geq 1$, and for any $b>0$ denote $T_b=\inf\{t\geq 1: S_t\geq b\}, T_b^{-}=\inf\{t\geq 1: S_t<(1-\alpha)\delta\,\, \text{or}\,\, S_t\geq b\}$, then by Lemmas \ref{stoptimejud} and \ref{ex548}, it has
\begin{equation}\label{a30}
  \textbf{E}\,T_b=\infty,\qquad \textbf{E}\,T_b^{-}<\infty.
\end{equation}
Since
\begin{equation}\label{a31}
\begin{split}
  P\{\max\limits_{1\le j\le t}S_j<b\}=P\{\max\limits_{1\le j\le t}S_j<(1-\alpha)\delta\}+P\{(1-\alpha)\delta\le \max\limits_{1\le j\le t}S_j<b\},
\end{split}
\end{equation}
we have
\begin{equation}\label{a32}
  \begin{split}
  \textbf{E}\,T_b=&\sum\limits_{t=1}^\infty P\{T_b\geq t\}=\sum\limits_{t=1}^\infty P\{\max\limits_{1\le j\le t-1}S_j<b\}\\
  =&\sum\limits_{t=1}^\infty\bigg(P\{\max\limits_{1\le j\le t-1}S_j<(1-\alpha)\delta\}+P\{(1-\alpha)\delta\le \max\limits_{1\le j\le t-1}S_j<b\}\bigg)\\
  =&\sum\limits_{t=1}^\infty P\{\max\limits_{1\le j\le t-1}S_j<(1-\alpha)\delta\}+\sum\limits_{t=1}^\infty P\{T_b^{-}\geq t\}\\
  =&\sum\limits_{t=1}^\infty P\{\max\limits_{1\le j\le t-1}S_j<(1-\alpha)\delta\}+\textbf{E}\,T_b^{-}.
  \end{split}
\end{equation}
By (\ref{a30}), it has
\begin{equation}\label{a33}
  \sum\limits_{t=1}^\infty P\{\max\limits_{1\le j\le t-1}S_j<(1-\alpha)\delta\}=\textbf{E}\,T_b-\textbf{E}\,T_b^{-}=\infty.
\end{equation}
Since $\{\max\limits_{1\le j\le t}S_j<(1-\alpha)\delta\}\subset\{\max\limits_{1\le j\le t}Q_j<0\}$ for $t\geq 1$, (\ref{a33}) yields
\begin{equation}\label{a34}
\begin{split}
  \textbf{E}\,T_0^{'}=&\sum\limits_{t=1}^\infty P\{T_0^{'}\geq t\}=\sum\limits_{t=1}^\infty P\{\max\limits_{1\le j\le t-1}Q_j<0\}\\
  \geq &\sum\limits_{t=1}^\infty P\{\max\limits_{1\le j\le t-1}S_j<(1-\alpha)\delta\}=\infty.
\end{split}
\end{equation}
This completes the proof.
\end{proof}

\end{document}